\documentclass[12pt, reqno]{amsart}
\usepackage{amsmath,amsopn,amssymb,amsthm}
\usepackage[breaklinks=true,colorlinks=true,linkcolor=blue,citecolor=blue,urlcolor=blue]{hyperref}

\newtheorem{thm}{Theorem}

\newtheorem{rmk}{Remark}

\DeclareMathOperator{\ad}{ad}
\DeclareMathOperator{\Ad}{Ad}

\DeclareMathOperator{\Ric}{Ric}

\begin{document}

\title[On a characterization of critical points \dots]
{On a characterization of critical points of the scalar curvature functional}
\author{Yu.G. Nikonorov}

\begin{abstract}
This is an English translation of the following paper, published several years ago:
{\it
Nikonorov Yu.G.
On a characterization of critical points of the scalar curvature functional (Russian),
Tr. Rubtsovsk. Ind. Inst.,  7, 211--217 (2000), Zbl 0956.53033}.
All inserted footnotes provide additional information related to the mentioned problem.

\vspace{2mm} \noindent 2020 Mathematical Subject Classification:
 53C25, 53C30, 17B20.

\vspace{2mm} \noindent Key words and phrases: critical point, Einstein metric, homogeneous Riemannian manifold, scalar curvature functional.
\end{abstract}

\maketitle

Let $G/H$ be a compact homogeneous spaces with connected  compact Lie group $G$ and its compact subgroup $H$.
Denote by $\mathfrak{g}$ and $\mathfrak{h}$ Lie algebras of $G$ and $H$ respectively.
We suppose that $G/H$ is almost effective, i.~e. there is no non-trivial ideals of the Lie algebra $\mathfrak{g}$  in $\mathfrak{h} \subset\mathfrak{g}$.
Denote by $B=B(\cdot\,,\cdot)$ the Killing form of $\mathfrak{g}$. Since $G$ is compact,
there is a positive definite bi-invariant inner product $\langle {\cdot}\,,{\cdot}\rangle$ on the Lie algebra $\mathfrak{g}$.
In the case when $G$ is semisimple, we can can choose $\langle{\cdot}\,,{\cdot}\rangle:=-B({\cdot}\,,{\cdot})$.

Let $\mathfrak{p}$ be the $\langle {\cdot}\,, {\cdot}\rangle$-orthogonal complement to $\mathfrak{h}$ in $\mathfrak{g}$.
It is clear that $\mathfrak{p}$ is $\Ad(H)$-invariant (and $\ad(\mathfrak{h})$-invariant, in particular).
The module $\mathfrak{p}$ is naturally identified with the tangent space to $G/H$ at the point $eH$.
Every $G$-invariant Riemannian metric on~$G/H$ generates an~$\Ad(H)$-invariant
inner product on~$\mathfrak{p}$ and vice versa. Therefore, it is possible to
identify invariant Riemannian metrics on~$G/H$ with $\Ad(H)$-invariant inner
products on~$\mathfrak{p}$ \cite{Bes}.
In the case when $G$ is semisimple, the Riemannian metric generated by the~inner product
$-B(\cdot\,,\cdot)\bigr\vert_{\mathfrak{p}}$ is called
{\it standard} or {\it Killing}.

We are interested in $G$-invariant
Einstein metrics on a homogeneous space
$M=G/H$.
Let us consider the set $\mathcal{M}_1$ of $\Ad(H)$-invariant inner product of some fixed volume on $\mathfrak{p}$.
It is well known that the critical points of the scalar curvature functional $S$ on $\mathcal{M}_1$
are exactly Einstein $G$-invariant metric on the space $G/H$ \cite{Bes, Jen2, Nik1998, WZ3}.

\smallskip

An interesting problem is the study of the nature of a given critical point
of the functional $S$. In particular, sometimes it is required to establish the nondegeneracy
of a critical point or check this point for the property of being a point of local
maximum (minimum) of the functional under investigation.\footnote{ See the paper
[{\it C. B\"{o}hm, M. Wang, W. Ziller, A variational approach for compact homogeneous Einstein manifolds,
Geom. Funct. Anal., 14:4 (2004), 681--733}\,] for a discussion of related problems.
A new approach to the study of the nature of critical points for $S$ is presented in the recent papers
[{\it J. Lauret, On the stability of homogeneous Einstein manifolds, preprint, 2021, arXiv:2105.06336},
{\it J. Lauret, C.E. Will, On the stability of homogeneous Einstein manifolds II, preprint, 2021, arXiv:2107.00354}].}

In this paper, we present some theorems giving partial answers
to the above questions.

\smallskip

Let us consider a metric $(\cdot , \cdot) \in \mathcal{M}_1$, which is a critical points of the scalar curvature functional $S$
and denote by $\Ric(\cdot , \cdot)$ its Ricci curvature.
Now, we consider a two-parameter variation of this metric in the class of
$\Ad(H)$-invariant metrics
$$
(\cdot , \cdot)_{t,s}=(e^{tV+sU}\cdot , \cdot),
$$
that is obtained using symmetric (with respect to $(\cdot ,\cdot)$) and $\Ad(H)$-equivariant operators $U$ and $V$ with zero traces.
The problem of studying the second variation of the scalar curvature functional is complicated by
by the fact that a priori it is impossible at the same time
bring three symmetric quadratic forms at once to a diagonal form.
In what follows, we assume that the forms
$(\cdot , \cdot)$, $(V\cdot , \cdot)$, $(U\cdot , \cdot)$
(they are $\Ad(H)$-invariant)
are simultaneously reduced to a diagonal form.
Let us consider an orthonormal basis for
$(\cdot , \cdot)$,
in which the operators $U$ and $V$ are diagonal.

In this case, the modules $\mathfrak{p}$ is decomposed into the sum of
$\Ad(H)$-irreducible and pairwise $(\cdot , \cdot)$-orthogonal submodules $\mathfrak{p}_1, \mathfrak{p}_2,\dots, \mathfrak{p}_q$.
In every $\mathfrak{p}_i$, we choose the ortonormal basis $e^1_i, e^2_i,...,e^{d_i}_i$ with respect to
$(\cdot, \cdot)$, where $d_i=\dim\, \mathfrak{p}_i$.
On every $\mathfrak{p}_i$, the quadratic forms
$(\cdot , \cdot)$, $(U\cdot ,\cdot)$, and $(V\cdot ,\cdot)$ are proportional each to other, sinse
$\Ad(H)$-irreducible.

For any triple of modules $\mathfrak{p}_i$, $\mathfrak{p}_j$, and $\mathfrak{p}_k$ we consider the value
$$
\left[
\begin{array}{c}
k
\\
i\,j
\\
\end{array}
\right]
=\sum\limits_{\alpha ,\beta,\gamma}
([e^{\alpha}_i,e^{\beta}_j]_p,e^{\gamma}_k)^2\, ,
$$
where $\alpha ,\beta,\gamma$ vary from $1$ to $d_i$, $d_j$,
and $d_k$
respectively and $[\cdot, \cdot]_{\mathfrak{p}}$ means the $\mathfrak{p}$-component of the Lie bracket.
Let us note that
$
\left[
\begin{array}{c}
k
\\
i\,j
\\
\end{array}
\right]
=
\left[
\begin{array}{c}
k
\\
j\,i
\\
\end{array}
\right]
$.
\smallskip

Let $r_i$ and $b_i$ be the values of the Ricci curvature and the Killing form
respectively on some unit vector from $\mathfrak{p}_i$. Obviously,
due to the irreducibility of this module,
these values are independent of the choice of the vector.
We know the following expression for the Ricci curvature (see e.g. (1) in \cite{Nik1998}) and for the scalar curvature (see \cite{WZ3})
of the metric $(\cdot , \cdot)$:
$$
r_k=-\frac{b_k}{2}
-\frac{1}{4d_k} \left( 2  \sum\limits_{i,j}
\left[
\begin{array}{c}
j
\\
k\,i
\\
\end{array}
\right]
-
\sum\limits_{i,j}
\left[
\begin{array}{c}
k
\\
i\,j
\\
\end{array}
\right] \right) \, ,
$$
$$
S=
-\sum\limits_{i} \frac{b_id_i}{2}
-\frac{1}{4} \sum\limits_{i,j,k}
\left[
\begin{array}{c}
k
\\
i\,j
\\
\end{array}
\right] \, .
$$

Obviously, if the metric $(\cdot, \cdot)$ is Einstein, then there exists a constant $r$ such that
the equality $r_k =r$ is fulfilled for all $k$.

Note that vectors
$\{e^{-\frac{u_it+v_is}{2}} e^k_i \}$, where $u_i$, $v_i$ are diagonal
elements of the operators $U$ and $V$ in the chosen basis, form an orthonormal
basis for the metric $(\cdot, \cdot)_{t,s}$.
Using formula 7.29 from \cite{Bes}, it is easy to obtain the following expression
for the scalar curvature of this metric:

$$
S((\cdot, \cdot )_{t,s})=
-\sum\limits_{i} \frac{b_id_i}{2}
e^{-u_i t-v_is}
-\frac{1}{4} \sum\limits_{i,j,k}
\left[
\begin{array}{c}
k
\\
i\,j
\\
\end{array}
\right]
e^{(u_k-u_i-u_j)t+(v_k-v_i-v_j)s} \, .
$$

It is easy to get

\begin{eqnarray*}
S((\cdot, \cdot )_{t,s})^{\prime}_t=
\sum\limits_{i} \frac{b_id_iu_i}{2}
e^{-u_i t-v_is}
-\frac{1}{4} \sum\limits_{i,j,k}
\left[
\begin{array}{c}
k
\\
i\,j
\\
\end{array}
\right]
(u_k-u_i-u_j)
e^{(u_k-u_i-u_j)t+(v_k-v_i-v_j)s} \,
\end{eqnarray*}

and

\begin{eqnarray*}
S((\cdot, \cdot )_{t,s})^{\prime \prime}_{ts}=
-\sum\limits_{i} \frac{b_id_iu_iv_i}{2}
e^{-u_i t-v_is}
\\
-\frac{1}{4} \sum\limits_{i,j,k}
\left[
\begin{array}{c}
k
\\
i\,j
\\
\end{array}
\right]
(u_k-u_i-u_j)(v_k-v_i-v_j)
e^{(u_k-u_i-u_j)t+(v_k-v_i-v_j)s} \, .
\end{eqnarray*}

If the metric $(\cdot, \cdot)$ is a degenerate critical point,
then there must exist a symmetric operator
$U$ with zero trace such that for any other symmetric operator
$V$ the equality
$S((\cdot, \cdot )_{t,s})^{\prime \prime}_{ts}(0,0)=0$
holds.

In particular, for all $V$ with the property
$\sum\limits_{i}v_id_i=0$ the following equality holds:

$$
\sum\limits_{i} \left(
-\frac{b_id_iu_i}{2}
-\frac{1}{4} \sum\limits_{j,k}
\left(
\left[
\begin{array}{c}
i
\\
k\,j
\\
\end{array}
\right]
(u_i-u_j-u_k)-
2\left[
\begin{array}{c}
k
\\
i\,j
\\
\end{array}
\right]
(u_k-u_i-u_j)
\right)
\right) v_i=0\, .
$$

Since $V$ is arbitrary, the last equality is equivalent to
the statement about the existence of such a number $\mu$,
that the equations

$$
-\frac{b_id_iu_i}{2}
-\frac{1}{4} \sum\limits_{j,k}
\left(
\left[
\begin{array}{c}
i
\\
k\,j
\\
\end{array}
\right]
(u_i-u_j-u_k)-
2\left[
\begin{array}{c}
k
\\
i\,j
\\
\end{array}
\right]
(u_k-u_i-u_j)
\right)
=\mu d_i
$$
hold  for all $i$.

\medskip

Now for the metric
$(\cdot, \cdot)$ and its fixed decomposition (any fixed decomposition of $\mathfrak{p}$ into the sum of pairwise $(\cdot, \cdot)$-orthogonal and $\Ad(H)$-irreducible
submodules) we define a square matrix $F$ of size $q \times q$, whose elements
are expressed as follows:

$$
f_{ii}=-\frac{b_id_i}{2}
-\frac{1}{4} \sum\limits_{j,k}
\left(
\left[
\begin{array}{c}
i
\\
k\,j
\\
\end{array}
\right]
+
2\left[
\begin{array}{c}
k
\\
i\,j
\\
\end{array}
\right]
\right)
-\frac{1}{2} \sum\limits_{j}
\left(
\left[
\begin{array}{c}
j
\\
i\,i
\\
\end{array}
\right]
-
2\left[
\begin{array}{c}
i
\\
i\,j
\\
\end{array}
\right]
\right),
$$

$$
f_{ij}=
\frac{1}{2} \sum\limits_{k}
\left(
\left[
\begin{array}{c}
j
\\
i\,k
\\
\end{array}
\right]
+
\left[
\begin{array}{c}
i
\\
k\,j
\\
\end{array}
\right]-
\left[
\begin{array}{c}
k
\\
i\,j
\\
\end{array}
\right]
\right) \mbox{\,\,\,\,\,for \,\,\, }i \neq j\,.
$$

\begin{thm}\label{thm0}
If $(\cdot , \cdot)$ is a degenerate critical point of the scalar curvature functional $S$
on the set $\mathcal{M}_1$, then the matrix $F$ is degenerate for some decomposition.
Conversely, if the matrix $F$ is degenerate for some decomposition and there are no pairwise isomorphic irreducible submodules in this (hence, any) decomposition,
then $(\cdot , \cdot)$ is a degenerate critical point of $S$ on the set $\mathcal{M}_1$.
\end{thm}

\begin{proof}
Let us consider the vector $d=(d_1,\dots ,d_q)\in\mathbb{R}^q$.
If the metric under consideration is a degenerate critical point,
then, as the previous reasoning shows,
there must exist such a decomposition, such a number $\mu$ and such a vector
$u = (u_1, \dots, u_q)$ with the condition
$\sum\limits_{i}d_iu_i =0$,
for which the equality
$Fu = \mu d$ holds.
It is easy to check that $Fb = rd$,
where all coordinates of the vector $b = (1, \dots, 1)\in \mathbb{R}^q$ are equal to $1$ and $r$ is the Ricci constant.
If $r = 0$, then $F$ is degenerate.
Suppose $r \neq 0$. It is clear that $F\bigl(\frac{\mu}{r}b\bigr)=\mu d$.
Since the vectors $b$ and $u$
are not collinear, and $F\bigr(\frac{\mu}{r}b-u\bigr)=0$, then the matrix $F$ is degenerate.

Let us now prove the second assertion of the theorem. We know that there are no
pairwise equivalent $\Ad(H)$-invariant irreducible submodules in  $\mathfrak{p}$.
Hence, any three quadratic $\Ad(H)$-invariant forms on $\mathfrak{p}$,
one of which is positive definite, are simultaneously diagonalized.
In this case, the matrix $F$ is uniquely determined.
It is enough for us to show the existence of the vector
$u=(u_1,...,u_q)$, that satisfies the conditions
$\sum\limits_{ i}d_iu_i=0$ and
$Fu=\mu d$.

Since the matrix $F$ is degenerate, there is a non-trivial vector
$p\in \mathbb{R}^q$ such that $Fp=0$. Moreover, we have
$Fb=rd$, where the vector $b=(1,\dots, 1)$ is as above.
Let us consider
the linear subspace $L_1$ in $\mathbb{R}^q$ spanned by the vectors $b$
and $p$. Obviously, we have $\dim (L_1) = 2$. Now we consider the hyperplane
$L_2=\left\{x\in \,|\,\sum\limits_{i=1}^q d_ix_i=0\right\}$.
Since
$\dim(L_2)=q-1$ and
$\dim(L_1)+\dim(L_2)=q+1>\dim(\mathbb{R}^q)=q$, then there is a non-trivial
vector $u\in L_1 \cap L_2$.
Obviously, it is what we need. Indeed, $Fu = \mu d$ for some real $\mu$, hence $\sum\limits_{ i} (Fu)_i v_i=0$ for any $v\in \mathbb{R}^q$ such that
$\sum\limits_{ i}d_iv_i=0$.
The theorem is proved.
\end{proof}

\medskip

We now turn to the study of sufficient conditions for the Einstein metric
(as the critical point of the scalar curvature functional $S$) would be the point
local minimum or maximum.

Lets us consider the following one-parameter variation of the metric $(\cdot, \cdot )$: $(\cdot, \cdot )_{t}=(\cdot, \cdot )_{t,s}|_{s=0}$.
Since
$$
S((\cdot, \cdot )_{t})^{\prime \prime}=S((\cdot, \cdot )_{t,s})^{\prime \prime}_{tt} (0,0)=
-\sum\limits_{i} \frac{b_id_iu_i^2}{2}
-\frac{1}{4} \sum\limits_{i,j,k}
\left[
\begin{array}{c}
k
\\
i\,j
\\
\end{array}
\right]
(u_k-u_i-u_j)^2 =
$$
$$
=\sum\limits_{i}
\left(
-\frac{b_id_i}{2}
-\frac{1}{4} \sum\limits_{j,k}
\left(
\left[
\begin{array}{c}
i
\\
k\,j
\\
\end{array}
\right]
+
2\left[
\begin{array}{c}
k
\\
i\,j
\\
\end{array}
\right]
\right)
-\frac{1}{2} \sum\limits_{j}
\left(
\left[
\begin{array}{c}
j
\\
i\,i
\\
\end{array}
\right]
-
2\left[
\begin{array}{c}
i
\\
i\,j
\\
\end{array}
\right]
\right)
\right) u_i^2+
$$

$$
+\sum\limits_{i,j}
\left(
\frac{1}{2} \sum\limits_{k}
\left(
\left[
\begin{array}{c}
j
\\
i\,k
\\
\end{array}
\right]
+
\left[
\begin{array}{c}
i
\\
k\,j
\\
\end{array}
\right]-
\left[
\begin{array}{c}
k
\\
i\,j
\\
\end{array}
\right]
\right)
\right)u_iu_j,
$$
then the question of whether the critical point under consideration gives
a local maximum (minimum) to the scalar curvature functional, essentially
is reduced to the question of negative (positive) definiteness
of the quadratic form generated by the matrix $F$ on the set of vectors
$u$ with the constraint
$\sum\limits_{i}d_iu_i=0$.

\smallskip

In what follows, we will investigate only the critical points corresponding to
{\it standard homogeneous Einstein metrics}
on compact homogeneous spaces $G/H$. Recall that the metric is called standard if
the corresponding inner product on $\mathfrak{p}$ is obtained
as a restriction of the minus Killing form of the algebra $\mathfrak{g}$.
There are many examples of standard homogeneous Einstein metrics \cite{Bes}.
Note that for the standard metric the expressions
$
\left[
\begin{array}{c}
k
\\
i\,j
\\
\end{array}
\right]
$
are symmetric in all three indices due to
the bi-invariance of the Killing form. Thus, the
expressions for the entries of the matrix $F$ simplified.
Moreover, we have $b_i=-1$ and
\begin{equation}\label{stwz}
\sum\limits_{j,k}
\left[
\begin{array}{c}
k
\\
i\,j
\\
\end{array}
\right]=
d_i(1-2c)
\end{equation}
for any $i=1,2,\dots, q$ by \cite{WZ3}, where
$c$ is {\it the Casimir constant}
\cite[Proposition 7.92]{Bes} of the standard homogeneous Einstein
manifold, that is related to the number $r$
via the relationship
$c = \frac{4r-1}{2}$.
It is well known \cite{Bes} that for any standard homogeneous
Einstein manifold, the inequalities
$\frac{1}{4} \leq r \leq \frac{1}{2}$ and
$0 \leq c \leq \frac{1}{2}$ hold.

Therefore, in the case of standard metrics, we have

$$
f_{ii}=\frac{d_i}{2}
-\frac{3}{4} \sum\limits_{j,k}
\left[
\begin{array}{c}
k
\\
i\,j
\\
\end{array}
\right]
+\frac{1}{2} \sum\limits_{j}
\left(
\left[
\begin{array}{c}
j
\\
i\,i
\\
\end{array}
\right]
\right)=
\frac{6c-1}{4}d_i
+\frac{1}{2} \sum\limits_{j}
\left(
\left[
\begin{array}{c}
j
\\
i\,i
\\
\end{array}
\right]
\right),
$$

$$
f_{ij}=
\frac{1}{2} \sum\limits_{k}
\left[
\begin{array}{c}
k
\\
i\,j
\\
\end{array}
\right].
$$

\begin{thm}\label{thm1}
Let $(\cdot, \cdot)$ be
a standard homogeneous Einstein metric, and
the Casimir constant satisfies the condition $c> 3/10$, then this
metric, as a critical point of the scalar curvature functional, is
a point of local minimum of the functional $S$ on the set $\mathcal{M}_1$.
\end{thm}

\begin{proof}
Taking into account (\ref{stwz}), we have the following inequality for the entries of the matrix $F$ under the condition $c>3/10$ for the standard Einstein metric:
$$
f_{ii}=\frac{6c-1}{4} d_i
+\frac{1}{2} \sum\limits_{j}
\left[
\begin{array}{c}
j
\\
i\,i
\\
\end{array}
\right]>
\frac{d_i}{2}(1-2c) +
\frac{1}{2} \sum\limits_{j}
\left[
\begin{array}{c}
j
\\
i\,i
\\
\end{array}
\right]\geq
$$
$$
\geq
\frac{1}{2} \sum\limits_{j,k}
\left[
\begin{array}{c}
k
\\
i\,j
\\
\end{array}
\right]\geq
\sum\limits_{j\neq i} f_{ij}=
\sum\limits_{j\neq i} |f_{ij}|.
$$
Therefore, for all indices  $i$ we get the inequalities
$$
f_{ii} >
\sum\limits_{j\neq i} |f_{ij}|,
$$
guaranteeing the positivity of the eigenvalues of the matrix $F$
even without the restriction
$\sum\limits_{i} d_iu_i=0$ by the Gershgorin circle theorem, see e.g. \cite[P.~390]{Gant}.\footnote{\,See also Theorems 6.1.1 and 6.1.10 in
[{\it Horn R.A., Johnson C.R. Matrix analysis. Second edition. Cambridge University Press, Cambridge, 2013\,}].}

The theorem is proved.\footnote{\,We can consider a symmetric space $G/H$ of compact type as an example to Theorem \ref{thm1}.
For such a symmetric space, the Casimir constant is equal to 1/2. Moreover, in this case the standard metric gives not only a local,
but a global minimum point of $S$ on $\mathcal{M}_1$.
Other more or less simple examples we can get among generalized Wallach spaces, see e.g.
[{\it Z. Chen, Yu.G. Nikonorov, Invariant Einstein metrics on generalized Wallach spaces, Science China Mathematics, 62:3 (2019), 569--584}] and the references therein.
If we consider the spaces from Table 1 of the  above paper, then the standard metric is Einstein if and only if $a_1=a_2=a_3=:a$ for this space
in the notation of Table 1. The Casimir constant of such standard Einstein metric is equal to $1/2-a$. Hence, if $a < 1/5$, then this standard Einstein metric gives
a local minimum point for $S$ on the set~$\mathcal{M}_1$. In particular, the standard metric is a local minimum point for $S$ for the spaces
$Sp(3k)/Sp(k)\times Sp(k)\times Sk(k)$ and $SU(3k)/S(U(k)\times SU(k) \times SU(k))$ when $k \geq 1$, as well as for the space $SO(3k)/SO(k)\times SO(k)\times SO(k)$
when $k \geq 5$. It could be shown also that the standard metric on the space $SO(3k)/SO(k)\times SO(k)\times SO(k)$ gives a local maximum point for $k=1$
(see Theorem \ref{thm2}),
a saddle point for $k=2$ and a local minimum point for $k=3,4$, see Theorem 7 in [{\it N.A. Abiev, A. Arvanitoyeorgos, Yu.G. Nikonorov, P. Siasos,
The Ricci flow on some generalized Wallach spaces,
Geometry and its Applications, Springer Proceedings in Mathematics $\&$ Statistics, Vol. 72, eds. V. Rovenski, P. Walczak, Springer, 2014, 3--37\,}].
The stability of Einstein metrics on all generalized Wallach and some more general spaces are studied in
[{\it J. Lauret, On the stability of homogeneous Einstein manifolds, preprint, 2021, arXiv:2105.06336},
{\it J. Lauret, C.E. Will, On the stability of homogeneous Einstein manifolds II, preprint, 2021, arXiv:2107.00354}].
}
\end{proof}

We now turn to the study of standard metrics on semisimples
compact Lie groups $G$. It is well known that such metrics are
Einstein (the Casimir constant in this case is $0$). We will give in some cases
their characterization from the point of view of analysis.

\begin{thm}\label{thm2}
Let
$(\cdot , \cdot)$
be the standard metric on a compact semisimple Lie group $G$. Then the following assertions hold.

1) If $G$ is locally isomorphic to the group $SU(2)$, then
$(\cdot , \cdot)$ is a local maximum point of the scalar curvature functional $S$.

2) For any semisimple group $G$, then the metric $(\cdot , \cdot)$ is not a local minimum point for $S$.

3) If the group $G$ is not simple, then
$(\cdot , \cdot)$ is a saddle point for the scalar curvature functional $S$.

4) If the group $G$ is locally isomorphic to the group $Sp(n)$ ($n\geq 2)$
or to the group $SU(n)$ ($n\geq 3$), then
$(\cdot , \cdot)$ is a saddle point for the scalar curvature functional $S$.

\end{thm}

\begin{proof}
Note that in the case of Lie groups we have $d_i=1$ and
$\left[
\begin{array}{c}
k
\\
i\,i
\\
\end{array}
\right]=0$
for all indices.

Let us first consider the case of a Lie group locally isomorphic to the group $SU(2)$.
If in the Lie algebra $su(2)$
we choose a basis orthonormal with respect to the standard
metrics, then we get the commutation relations
$[e_i,e_j]=\pm \frac{1}{\sqrt{2}}e_k$ for $i\neq j \neq k \neq i$ ($\{i,j,k\}=\{1,2,3\}$).
Hence,  we get
$\left[
\begin{array}{c}
k
\\
i\,j
\\
\end{array}
\right]=1/2$ for distinct indices.
The $(3\times 3)$-matrix $F$ is such that $f_{ii}=-1/4$ and
$f_{ij}=1/4$ for $i \neq j$.
This matrix has the eigenvalues $1/4$ (of multiplicity $1$) and $-1/2$ (of multiplicity $2$), moreover
the eigenvalue $1/4$ corresponds to the eigenvector $(1,1,1)$.
Therefore, on the set of vectors $u=(u_1,u_2,u_3)$ with the relation
$u_1+u_2+u_3=0$,
the matrix $F$ is negatively definite, that is, the standard metric
is a point of strict local maximum for $S$. This proves assertion 1) of the theorem.
\smallskip

Consider now an arbitrary semisimple compact Lie group $G$.
with some closed subgroup $K$.
We represent the Lie algebra $\mathfrak{g}$ of the Lie group $G$ in the form
$\mathfrak{g}=\mathfrak{p}_1 \oplus \mathfrak{p}_2$, where the module $\mathfrak{p}_1$ is the Lie subalgebra $\mathfrak{k}$,
and the module $\mathfrak{p}_2$ is orthogonal to $\mathfrak{p}_1$
with respect to the standard metric (or, equivalently, with respect to the Killing form of $\mathfrak{g}$).
Let us consider the metrics, that have the following form:
$$
(\cdot, \cdot)_1=
x_1(\cdot , \cdot)|_{\mathfrak{p}_1}+
x_2(\cdot , \cdot)|_{\mathfrak{p}_2}, \quad \mbox{where} \quad x_1, x_2 >0.
$$
For the standard metric, the following relations holds:
$$\left[
\begin{array}{c}
1
\\
1\,1
\\
\end{array}
\right]+
\left[
\begin{array}{c}
2
\\
1\,2
\\
\end{array}
\right]=d_1, \quad
\left[
\begin{array}{c}
2
\\
1\,2
\\
\end{array}
\right]=
\left[
\begin{array}{c}
2
\\
2\,1
\\
\end{array}
\right]=
\left[
\begin{array}{c}
1
\\
2\,2
\\
\end{array}
\right],
$$

$$\left[
\begin{array}{c}
2
\\
2\,1
\\
\end{array}
\right]+
\left[
\begin{array}{c}
1
\\
2\,2
\\
\end{array}
\right]+
\left[
\begin{array}{c}
2
\\
2\,2
\\
\end{array}
\right]=d_2,
$$
where $d_i$ means the dimension of the module $p_i$. If we introduce the notation
$$
a:=\left[
\begin{array}{c}
1
\\
2\,2
\\
\end{array}
\right],
$$
we get a rather simple formula for the scalar curvature
of considered metrics, namely
$$
S((\cdot , \cdot)_1)=
\frac{d_1-a}{4} \frac{1}{x_1}+
\frac{d_2+2a}{4} \frac{1}{x_2}-
\frac{a}{4} \frac{x_1}{x_2^2}.
$$
Takin into account the volume restriction
$x_1^{d_1}x_2^{d_2}=1$, we consider a one-parameter variation of the standard metric of the form
$x_1=t$, $x_2=t^{-d_1/d_2}$.
Then
it is enough for us to investigate for the extremum at the point $t = 1$
the function
$$
f(t)=(d_1-a)t^{-1}+(d_2+2a)t^{d_1/d_2}-at^{1+2d_1/d_2}.
$$
It is easy to check that
$$
f^{\prime}(t)=-(d_1-a)t^{-2}+\frac{d_1}{d_2}(d_2+2a)t^{d_1/d_2-1}-
\left(
1+2\frac{d_1}{d_2}
\right)
ct^{2d_1/d_2},
$$
$$
f^{\prime \prime}(1)=2(d_1-a)+\frac{d_1}{d_2}
\left(
\frac{d_1}{d_2}-1
\right)
(d_2+2a)-
2\frac{d_1}{d_2}
\left(
1+2\frac{d_1}{d_2}
\right)a.
$$
Therefore, we will be interested in the sign of the number
$$
T=
d_2^2f^{\prime \prime}(1)=(d_1+d_2)(d_1d_2-2(d_1+d_2)a).
$$
\smallskip

Since we can take a one-dimensional torus as a subgroup $K$,
for which $d_1=a=1$,
we find a one-dimensional variation, for which
$T <0$ and scalar curvature functional $S$
reaches a local maximum on the standard metric (for $t=1$).
This proves assertion 2) of the theorem.
\smallskip

If the group $G$ is not simple, then as a subgroup
$K$ one can take a proper normal subgroup.
In this case, the modules $\mathfrak{p}_1$ and $\mathfrak{p}_2$ are ideals in the Lie algebra $\mathfrak{g}$.
Thus,
$
a=\left[
\begin{array}{c}
1
\\
2\,2
\\
\end{array}
\right]=0
$.
Therefore,
$T=(d_1+d_2)d_1d_2>0$, and the standard metric (due to assertion 2)) is a saddle point of the scalar curvature functional $S$.
\smallskip

Now let $G$ be locally isomorphic to the group $Sp(n)$ for $n\geq 2$.
As a subgroup of $K$, we consider $Sp(n-1)$.
In this case, we have the relations
$d_1=2n^2-3n+1$, $d_2=4n-1$,
$
\left[
\begin{array}{c}
1
\\
1\,1
\\
\end{array}
\right]=\frac{n}{n+1}d_1$,
hence,
$
a=\left[
\begin{array}{c}
1
\\
2\,2
\\
\end{array}
\right]=\frac{2n^2-3n+1}{n+1}
$.
Therefore,
$$
T=(2n^2+n)(2n^2-3n+1)\left(
4n-1-2\frac{2n^2+n}{n+1}
\right)=\frac{(2n^2+n)(2n-1)(n-1)^2}{n+1}>0,
$$
that is, for $n \geq 2$, the standard metric on the group $Sp(n)$
is a saddle point of the functional $S$ (due to assertion 2)).
\smallskip

It remains for us to consider the group $SU(n)$ for $n\geq 3$.
Let us consider a subgroup $U(n-1)$ of $SU(n)$.
We get the decomposition $su(n)=\mathfrak{g} = \mathfrak{p}_1 \oplus \mathfrak{p}_2 \oplus \mathfrak{p}_3$, where the module $\mathfrak{p}_1$ corresponds to
the subalgebra $\mathbb{R}$, $\mathfrak{p}_2$ corresponds to the subalgebra $su(n-1)$
($\mathfrak{p}_1 \oplus\mathfrak{p}_2$ is the Lie algebra of the subgroup $U(n-1)$), and the module $\mathfrak{p}_3$
is orthogonal to $\mathfrak{p}_1 \oplus \mathfrak{p}_2$ with respect to the Killing form of $\mathfrak{g}$.
Since the pair $(su (n), su(n-1)\oplus \mathbb{R})$ is symmetric,
then it is easy to show that
$d_1=1$, $d_2=n^2-2n$, $d_3=2n-2$,
$
\left[
\begin{array}{c}
1
\\
3\,3
\\
\end{array}
\right]=1
$,
$
\left[
\begin{array}{c}
2
\\
3\,3
\\
\end{array}
\right]=n-2
$.
Let us consider the following one-parameter family of metrics:
$$
(\cdot , \cdot )_t=
e^{2t}(\cdot ,\cdot)|_{\mathfrak{p}_1}+
e^{-2t/(n-2)}(\cdot ,\cdot)|_{\mathfrak{p}_2}+
e^{t}(\cdot ,\cdot)|_{\mathfrak{p}_3}.
$$
Since
$2d_1-\frac{2}{n-2}d_2+d_3=0$,
then the metrics under consideration have the same volume as the standard one has.
It is easy to calculate the scalar curvature for this family of metrics.
$$
h(t)=S((\cdot , \cdot )_t)=
\frac{(n-1)(n-2)}{4}\,e^{2t/(n-2)}+(n-1)e^{-t}-\frac{1}{4}-
\frac{n-2}{4}e^{-2(n-1)t/(n-2)}.
$$
Therefore,
$$
h^{\prime}(t)=
\frac{n-1}{2}\,e^{-2(n-1)t/(n-2)}
{\left(
e^{nt/(n-2)}-1
\right)}^2 >0
$$
for $t\neq 0$, hence, the function $h$ strictly increases.
Therefore, in this case, the standard metric is also a
saddle point for the scalar curvature functional $S$.
The theorem is proved. \footnote{\,It should be noted that the standard metric on any Lie group locally isomorphic to the group $SO(n)$, $n=4,5,6$,
is a saddle point (indeed, $SO(4)$ is not simple,
$SO(5)$ is locally isomorphic to $Sp(2)$, and $SO(6)$ is locally isomorphic to $SU(4)$). On the other hand, it is quite unexpectedly
that the standard metric on any Lie group locally isomorphic to the group $SO(n)$, $n\geq 7$, as well as on any exceptional compact simple Lie group, gives a local maximum
for the scalar curvature functional $S$,
see Section 5.1 in [{\it J. Lauret, On the stability of homogeneous Einstein manifolds, preprint, 2021, arXiv:2105.06336}\,] for the proof.}
\end{proof}

\begin{rmk}
It is easy to show that the standard metric on the group
$SU(2)$ is not a global maximum point
of the scalar curvature functional $S$.
Note also that this
the statement is a simple consequence of one of the results
by M.~Wang and W.~Ziller~\cite{WZ3}.
\end{rmk}

\begin{rmk}
The above construction for groups $SU(n)$, $n \geq 3$, generalizes the construction
by G.~Jensen for the group $SU(3)$ \cite{Jen2}.
Note that $h^{\prime \prime}(0)=0$ in this case.
\end{rmk}


\vspace{3mm}
\begin{thebibliography}{111}
\bibitem{Bes}
\textsc{A.L. Besse}, Einstein Manifolds, Springer-Verlag, 1987.
\bibitem{Gant}
\textsc{F.R. Gantmakher},
Theory of matrices. 4th ed., compl. Ed. by V. B. Lidskij (Russian),
Moskva: Nauka, 1988.
\bibitem{Jen2}
\textsc{G.R. Jensen}, The scalar curvature of left  invariant
Riemannian metrics, Indiana Univ. Math. J., 1971, V. 20, P.~1125--1143.
\bibitem{Nik1998}
\textsc{Yu.G. Nikonorov}, The scalar curvature functional and homogeneous Einstein metrics on Lie groups, Siberian Math. J., 1998, V.~39, N~3, P.~504--509.
\bibitem{WZ3}
\textsc{M.~Wang, W.~Ziller}, Existence and  Non-existence of
Homogeneous Einstein metrics, Invent. Math., 1986,  V. 84, P.~177--194.
\end{thebibliography}
\end{document}